\documentclass[12pt,english,letterpaper]{amsart}
\usepackage[T1]{fontenc}
\usepackage[margin=1in]{geometry}
\usepackage{amsmath,amscd}
\usepackage{amssymb,mathtools}
\usepackage{amsthm}
\usepackage{yfonts}
\usepackage{mathrsfs}
\usepackage{MnSymbol}
\usepackage{pictexwd,dcpic}
\usepackage{xfrac}
\usepackage{babel}
\usepackage{enumerate}
\usepackage{hyperref}
\hypersetup{
    colorlinks=true,
    linkcolor=blue,
    filecolor=magenta,      
    urlcolor=cyan,
    citecolor=magenta
}
\usepackage{cleveref}
\usepackage[normalem]{ulem}
\usepackage{mdwlist} 

\newtheorem{thm}{Theorem}[section]
\newtheorem{cor}[thm]{Corollary}
\newtheorem{lem}[thm]{Lemma}

\theoremstyle{definition}
\newtheorem*{defn}{Definition}
\newtheorem{conj}[thm]{Conjecture}

\theoremstyle{remark}
\newtheorem{remark}[thm]{Remark}

\title{Modular Zilber--Pink for Geometrically Generic Varieties}
\author[V. Aslanyan]{Vahagn Aslanyan}
\email{vahagn.aslanyan@manchester.ac.uk}
\address{Department of Mathematics, University of Manchester, Manchester, UK}
\author[S. Eterovi\'c]{Sebastian Eterovi\'c}
\email{sebastian.eterovic@univie.ac.at}
\address{Kurt G\"odel Research Center, Universit\"at Wien, 1090 Wien, Austria}
\author[G. Fowler]{Guy Fowler}
\email{guy.fowler@manchester.ac.uk}
\address{Department of Mathematics, University of Manchester, Manchester, UK\newline
	\indent Heilbronn Institute for Mathematical Research, Bristol, UK}

\date{}

\thanks{V.A. was supported by EPSRC Open Fellowship EP/X009823/1. S.E.~was supported by the Austrian Science Fund (FWF) 10.55776/ESP1584024. 
The authors wish to thank the referee for their careful reading and detailed comments on an earlier draft of this work, and for suggesting several improvements to our arguments.
For the purpose of open access, the authors have applied a Creative Commons Attribution (CC BY) licence to any Author Accepted Manuscript version arising from this submission.}



\begin{document}

\begin{abstract}
    We prove that the modular Zilber--Pink conjecture (in Pink's formulation in terms of unlikely intersections) holds for all subvarieties $V$ of $ \mathrm{Y}(1)^n$ for which no projection to any $\dim V + 2$ coordinates is defined over the algebraic numbers.
\end{abstract}

\maketitle

\section{Introduction}
\label{sec:intro}
The modular curve $\mathrm{Y}(1):=\mathrm{SL}_2(\mathbb{Z})\backslash\mathbb{H}$ has the structure of a coarse moduli space for complex elliptic curves. 
Using the modular $j$-function, we can identify $\mathrm{Y}(1)$ with the complex affine line $\mathbb{A}_{\mathbb{C}}^1$, thus realizing $\mathrm{Y}(1)$ as an algebraic variety. 
The cartesian powers of $\mathrm{Y}(1)$ can then be viewed as moduli spaces for tuples of complex elliptic curves. 
Some subvarieties of $\mathrm{Y}(1)^n$ are compatible with the moduli structure on $\mathrm{Y}(1)^n$, and as such they are deemed to be \emph{special} (we will give a precise definition in \S\ref{subsec:specials}). 

Given a positive integer $n$ and $k \in \{0,\ldots,n\}$, let $\mathscr{S}_n^k$ denote the set of all special subvarieties of $\mathrm{Y}(1)^n$ of dimension at most $k$.
In this paper we prove the following \emph{unlikely intersections} result for sufficiently generic subvarieties of $\mathrm{Y}(1)^n$. 
Unless otherwise stated, in this paper we will assume all algebraic varieties to be irreducible.

\begin{thm}
\label{thm:genericmzp}
    Let $n$ be any positive integer. 
    Let $V\subseteq \mathrm{Y}(1)^n$ be an algebraic variety of dimension $d$. 
    Suppose that for every coordinate projection $\mathrm{pr}: \mathrm{Y}(1)^n\to \mathrm{Y}(1)^{d+2}$ to any $d+2$ coordinates, $\mathrm{pr}(V)$ is not definable over $\overline{\mathbb{Q}}$. 
    If $V$ is not contained in a proper special subvariety of $\mathrm{Y}(1)^n$, then
    \[ \bigcup_{S \in \mathscr{S}_n^{n- (d+1)}} V \cap S\]
    is not Zariski dense in $V$. 
\end{thm}

In particular, when $\dim V = n - 2$, the genericity condition reduces to $V$ not being defined over $\overline{\mathbb{Q}}$.
When $\dim V \geq n -1$, the genericity condition is empty, but for such $V$ the conclusion of the theorem is already known.

In terms of the modular Zilber--Pink conjecture (which will be recalled in various formulations in \S\ref{subsec:mzp}), this theorem says that, under the conditions of the theorem, the \emph{unlikely locus} of $V$ is not Zariski dense in $V$ (the definition of the unlikely locus will be given in \S\ref{subsec:mzp}).
 
The Zilber--Pink conjecture is a major topic of research in arithmetic geometry. 
In this paper we will only consider the modular version of Zilber--Pink (that is, the version concerning subvarieties of $\mathrm{Y}(1)^n$), but the conjecture has been formulated in many other settings, including (semi) abelian varieties, (mixed) Shimura varieties, and more generally variations of (mixed) Hodge structures (see \cite{daw:unlikelyintersectionsshimuravarieties,klingler:hodgetheoryicm,pila_2022} for various surveys).
Even in the modular case, very few instances of this conjecture have been proven, with the known cases usually either focusing on the intersection of $V$ with special varieties of dimension zero (the Andr\'e--Oort conjecture, see \S\ref{subsec:specials}), or considering $V$ to be a curve (see e.g.~\cite{daw-orr:zpproductmodularcurves,habegger-pila,papas:zpnmultiplicativedeg}).

Some aspects of Zilber--Pink are arithmetic in nature, and in some sense therein lies the main difficulty when approaching this conjecture. 
To be more precise, the results of Habegger and Pila in \cite{habegger-pila:o-min} (later extended to other Shimura varieties by Daw and Ren in the case that $V$ is a curve in \cite{daw-ren} and in the PhD thesis of Cassani in general, see \cite[\S6.2]{daw:unlikelyintersectionsshimuravarieties}) show that, using the Pila--Zannier strategy, one can reduce the Zilber--Pink conjecture to a problem of point-counting.
However, the point-counting required using this method depends on the Large Galois Orbits conjecture (see \cite[Conjecture 8.2]{habegger-pila:o-min} for the modular case), of which few cases have been established. 

The conditions of Theorem \ref{thm:genericmzp} allow us to prove modular Zilber--Pink for as many positive dimensional subvarieties of $\mathrm{Y}(1)^n$ as possible, short of proving it for subvarieties of codimension $\geq 2$ defined over number fields where deep arithmetic inputs are expected to be needed.
Indeed, if there is an algebraic variety $V\subseteq\mathrm{Y}(1)^n$ of dimension $0<d\leq n-2$ and there is some coordinate projection $\mathrm{pr}:\mathrm{Y}(1)^n\to\mathrm{Y}(1)^{d+2}$ such that $\mathrm{pr}(V)$ is defined over $\overline{\mathbb{Q}}$, then we can apply Corollary \ref{cor:prunlikely} below to conclude that proving modular Zilber--Pink for $V$ would imply modular Zilber--Pink for $\mathrm{pr}(V)$. 

As we mentioned earlier, Theorem \ref{thm:genericmzp} is connected with one of the formulations of modular Zilber--Pink having to do with the unlikely locus of $V$. 
Other formulations of Zilber--Pink (which we will recall) speak about maximal atypical components or optimal subvarieties, and we show in Remark \ref{remark:atypical} a consequence of Theorem \ref{thm:genericmzp} in terms of these concepts.

\subsection{Related results}
\label{subsec:relatedresults}

The case of Theorem \ref{thm:genericmzp} where $V$ has dimension 1 (i.e.~$V$ is an algebraic curve) was proven by Pila in his modular Zilber--Pink result for generic curves \cite[Theorem 1.4]{pila:fermat}.
That proof follows the Pila--Zannier strategy and uses o-minimality, point-counting, and gonality estimates. 
As we explain later, our proof of Theorem \ref{thm:genericmzp} uses different techniques, thus giving in particular a new proof of Pila's result.

The non-definability condition in Theorem~\ref{thm:genericmzp} is similar in nature to the conditions required in the results present in \cite[Corollary 1.1]{pila-scanlon}, \cite[Appendix A]{barroero-dill:heckeorbits} and \cite{klingler-tayou:zilberpink}, all of which establish some form of generic Zilber--Pink result, but as we explain below, Theorem \ref{thm:genericmzp} is not a special case of any of them.

The results of Barroero and Dill in \cite[Appendix A]{barroero-dill:heckeorbits} are restricted to Shimura data whose associated $\mathbb{Q}$-algebraic group has (non-trivial) simple adjoint group, so it does not cover higher-powers of the corresponding Shimura varieties, like $\mathrm{Y}(1)^n$ with $n>1$. 
In particular, if one wanted to use \cite[Appendix A]{barroero-dill:heckeorbits} for powers of the modular curve, one could only use it for subvarieties of $\mathrm{Y}(1)$, in which case the modular Zilber--Pink conjecture becomes trivial. 

The results of Barroero and Dill were extended by Klingler and Tayou in \cite{klingler-tayou:zilberpink}, and in particular the results of the latter can be applied to powers of Shimura varieties. 
The condition ``no non-trivial images are defined over $\overline{\mathbb{Q}}$'' required by \cite[Theorem 1.6 and Corollary 1.7]{klingler-tayou:zilberpink} when specified to subvarieties of the Shimura datum of $\mathrm{Y}(1)^n$ amounts to the following condition: given a subvariety $V\subseteq \mathrm{Y}(1)^n$, every projection of $V$ to $\mathrm{Y}(1)$ is either constant or not defined over $\overline{\mathbb{Q}}$. 
However, the projection of $V$ to $\mathrm{Y}(1)$ is either dominant or constant, so the only case that falls under the conditions of \cite[Theorem 1.6]{klingler-tayou:zilberpink} is when every projection of $V$ to $\mathrm{Y}(1)$ is constant, meaning that $V$ is a point, but the modular Zilber--Pink conjecture again becomes trivial in this case. 

Although the proof we give of Theorem \ref{thm:genericmzp} is inspired by the work present in \cite{pila-scanlon}, our theorem also differs from \cite[Corollary 1.1]{pila-scanlon} as the latter only focuses on the subset of the unlikely locus of points none of whose coordinates are algebraic.\footnote{\cite{pila-scanlon} actually focuses on the subset of the unlikely locus obtained by intersections with the proper subset of $\mathscr{S}_{n}^{n-(\dim V+1)}$ of \emph{strongly special} varieties (we recall the definition in \S\ref{subsec:specials}). 
However, restricting to strongly special varieties does not actually impose any further restrictions for the result we cite, for if a point has no algebraic coordinates then it cannot belong to a non-strongly special variety.}

\subsection{Aspects of the proof}
We use techniques from differential algebra, model theory, and some arithmetic input in the form of the modular Mordell--Lang theorem proven by Habegger and Pila (see \S\ref{subsec:mordell-lang}).
For the most part, the strategy for proving Theorem \ref{thm:genericmzp} is based on the work of Pila and Scanlon in \cite{pila-scanlon}. 
As such, the proof of Theorem \ref{thm:genericmzp} differs greatly in nature from the methods used in \cite{barroero-dill:heckeorbits} and \cite{klingler-tayou:zilberpink}, although functional transcendence results (namely, Ax--Schanuel) are a key ingredient in all proofs.

We expect that our work generalizes to other settings where a corresponding Ax--Schanuel theorem and a suitable Mordell--Lang theorem are known, such as Shimura varieties of abelian type, see \cite[Theorem 3.10]{aslanyan-daw} (Ax--Schanuel being provided by \cite{mpt}). 
We give more references in \S\ref{subsec:mordell-lang}.

\subsection{Structure of the paper}
In \S \ref{sec:background} we cover the important definitions related to the various formulations of the modular Zilber--Pink conjecture (\S\ref{subsec:mzp}), and we also recall crucial results which lead up to the modular Mordell--Lang theorem (\S\ref{subsec:mordell-lang}).
We also recall the differential formulation of the Ax--Schanuel theorem for $j$ (\S\ref{subsec:ax-schanuel}), as well as some notions from the model theory of differentially closed fields having to do with generic solutions (\S\ref{subsec:dcf}). 
The proof of the main theorem is then given in \S\ref{sec:main}.

\section{Background}
\label{sec:background}

This section is entirely devoted to standard definitions and results relevant to the proof of our main theorem. 
The reader who is familiar with the material can go to \S\ref{sec:main} and refer back to this section as needed. 

\subsection{Special varieties and modular Andr\'e--Oort}
\label{subsec:specials}

We denote by $\mathbb{H}$ the complex upper half-plane. 
The modular $j$-function $j:\mathbb{H}\to \mathbb{C}$ allows us to identify $\mathrm{SL}_2(\mathbb{Z}) \backslash \mathbb{H}\simeq \mathbb{C}$. 
From now on we identify $\mathrm{Y}(1)$ with the complex affine line $\mathbb{A}_\mathbb{C}^1$.

Let $\left\{\Phi_{N}(X,Y)\right\}_{N=1}^{\infty}\subseteq\mathbb{Z}[X,Y]$ denote the family of \emph{modular polynomials} associated with $j$ (see \cite[Chapter 5, Section 2]{lang:elliptic} for the definition and main properties of this family). 

Although we just identified $\mathrm{Y}(1)$ with $\mathbb{A}_{\mathbb{C}}^1$, in the proof of the main theorem we will consider other algebraically closed fields $K$ of characteristic zero, so we have opted to state some of the following definitions for any such $K$.

\begin{defn}
A finite set $A\subseteq K$ is said to be \emph{modularly independent} if for every pair of distinct numbers $a,b$ in $A$ and every positive integer $N$, we have that $\Phi_{N}(a,b)\neq 0$. 
Otherwise, we say that $A$ is \emph{modularly dependent}.

An element $w\in K$ is said to be \emph{modularly dependent over $A$} if either $w\in A$ or there is $a\in A$ such that the set $\{a,w\}$ is modularly dependent.
Otherwise we say $w$ is \emph{modularly independent over $A$}. 
\end{defn}

\begin{defn}
A point $w$ in $\mathbb{C}$ is said to be \emph{special} (also known as a \emph{singular modulus}) if there is $z$ in $\mathbb{H}$ such that $[\mathbb{Q}(z):\mathbb{Q}] = 2$ and $j(z)=w$. 
Special points have the property that they are algebraic integers and that every Galois conjugate of a special point is also a special point \cite[Proposition~25]{zagier:elliptic}.
Using this, we define a special point of $K$ to be any root in $K$ of the minimal polynomial (over $\mathbb{Q}$) of a special point in $\mathbb{C}$.

We further say that a point $\mathbf{w}$ in $\mathbb{A}^{n}_K$ is \emph{special} if every coordinate of $\mathbf{w}$ is special.
We denote the set of special points of $K$ by $\mathscr{S}^0_1(K)$. 
\end{defn}

\begin{defn}
A \emph{special subvariety} of $\mathbb{A}_{K}^{n}$ is an irreducible component of an algebraic set defined by equations of the following forms: 
\begin{enumerate}[(a)]
    \item  $\Phi_{N}(X_{i},X_{k}) = 0$, for some $N\in\mathbb{N}$, and
    \item $X_{i} = s$, where $s\in K$ is a special point. 
\end{enumerate}
We allow the set of equations to be empty, so $\mathbb{A}^{n}_{K}$ is itself a special variety. 
A special subvariety on which no coordinate has a fixed value is called a \emph{strongly special subvariety}.
\end{defn}

Every special point in $K$ is a solution of an equation $\Phi_N(X,X)=0$ for some $N\in\mathbb{N}_{\geq 2}$, hence condition (b) in the previous definition is, strictly speaking, superfluous, but we state it to emphasise the analogy with weakly special subvarieties (defined below).

Every special variety has a Zariski dense set of special points (see e.g.~\cite[1.4 Aside]{pila:andre-oort}). 
The Andr\'e--Oort conjecture, proved by Pila, establishes the converse.

\begin{thm}[Andr\'e--Oort, {{\cite{pila:andre-oort}}}]
    \label{thm:andre-oort}
    Let $X\subseteq\mathbb{A}_{\mathbb{C}}^n$ be a non-empty collection of special points. 
    Then every irreducible component of the Zariski closure of $X$ is a special variety. 
\end{thm}

\begin{defn}
A subvariety of $\mathbb{A}_{K}^{n}$ is called \emph{weakly special} if it is an irreducible component of an algebraic set defined by equations of the following forms: 
\begin{enumerate}[(i)]
    \item $\Phi_{N}(X_{i},X_{k}) = 0$, for some $N\in\mathbb{N}$,
    \item $X_{\ell} = d$, for some constant $d\in K$.
\end{enumerate}
 In particular, a weakly special subvariety is special if and only if it contains a special point.
\end{defn}

We remark that the irreducible components of any non-empty intersection of (weakly) special subvarieties are again (weakly) special subvarieties. 
This allows us to make the following definition.

\begin{defn}
Given an irreducible constructible set $X\subseteq\mathbb{A}^n_{K}$, we denote by $\operatorname{spcl}(X)$ the \emph{special closure} of $X$, that is, the smallest special subvariety containing $X$. 
Similarly, we denote by $\mathrm{wspcl}(X)$ the \emph{weakly special closure} of $X$.
\end{defn}

Now we introduce some notation for coordinate projections. 
Let $n$ and $\ell$ denote positive integers with $\ell \leq n$, and let $\mathbf{i}=(i_1,\ldots,i_\ell)$ denote a point in $\mathbb{N}^{\ell}$ with $1\leq i_1 < \ldots < i_\ell \leq n$. 
Define the projection map $\mathrm{pr}_{\mathbf{i}}:\mathbb{A}^{n} \rightarrow \mathbb{A}^{\ell}$ by
\[\mathrm{pr}_{\mathbf{i}}:(x_1,\ldots,x_n)\mapsto (x_{i_1},\ldots,x_{i_\ell}).\] 
In particular we distinguish between the natural number $\ell$ and the tuple $\boldsymbol{\ell}=(1,\ldots,\ell)$. 

\begin{remark}
If $T$ is a (weakly) special subvariety of $\mathbb{A}^{n}_K$, then for any choice of indices $1\leq i_{1}<\cdots<i_{\ell}\leq n$ we have that $\mathrm{pr}_{\mathbf{i}}(T)$ is a (weakly) special subvariety of $\mathbb{A}^{\ell}_K$. 
\end{remark}

\subsection{Modular Zilber--Pink}
\label{subsec:mzp}

Suppose that $V$ and $W$ are subvarieties of a smooth algebraic variety $Z$ with a non-empty intersection. 
Let $X$ be an irreducible component of the intersection $V\cap W$. 
We say that $X$ is an \emph{atypical component of $V\cap W$ (in $Z$)} if
\[\dim X > \dim V + \dim W - \dim Z.\]
We say that the intersection \emph{$V\cap W$ is atypical (in $Z$)} if it has an atypical component. 
Otherwise, we say that \emph{$V\cap W$ is typical (in $Z$)}, i.e.~$V\cap W$ is typical in $Z$ if $\dim V\cap W = \dim V + \dim W - \dim Z$.

As before, throughout $K$ will denote an algebraically closed field of characteristic zero.

\begin{defn}
Given a subvariety $V \subseteq \mathbb{A}^{n}_K$, we say that $X$ is an \emph{atypical component of $V$} if there exists a special subvariety $S$ of $\mathbb{A}^{n}_K$ such that $X$ is an atypical component of $V\cap S$. 
We remark that in this case, since $\dim S\geq \dim\operatorname{spcl}(X)$, it is also true that $X$ is an atypical component of $V\cap \operatorname{spcl}(X)$.

We say that $X$ is a \emph{strongly atypical component of $V$} if $X$ is an atypical component of $V$ and no coordinate is constant on $X$. 

An atypical (resp.~strongly atypical) component of $V$ is said to be \emph{maximal} (in $V$) if it is not properly contained in another atypical (resp.~strongly atypical) component of $V$.  

We define the \emph{unlikely locus of $V$} to be the set of all points $\mathbf{x}\in V$ for which there is a special variety $S$ with $\dim V+ \dim S<n$ such that $\mathbf{x}\in V\cap S$. 
Note that unlikely intersections are atypical, hence the unlikely locus is contained in the union of all atypical components
\end{defn}

\begin{defn}
Let $V$ be a subvariety of $\mathbb{A}^{n}_K$. 
Given a subvariety $X\subseteq V$, we define the \emph{defect of $X$} to be
\[\mathrm{def}(X)\coloneqq \dim\operatorname{spcl}(X) - \dim X.\]
We say that $X$ is \emph{optimal in $V$} if for every subvariety $W\subseteq V$ satisfying $X\subsetneq W$ we have that $\mathrm{def}(X)<\mathrm{def}(W)$. 
We let $\operatorname{Opt}(V)$ denote the set of all optimal subvarieties of $V$. 
Observe that always $V\in\operatorname{Opt}(V)$. 
\end{defn}

\begin{remark}
\label{rem:maxatypisopt}
A maximal atypical component of $V$ is optimal in $V$. 
However, optimal subvarieties need not be maximal atypical.

On the other hand, if $X$ is a proper subvariety of $V$ which is optimal in $V$, then $\mathrm{def}(X) < \mathrm{def}(V)$ which implies that $\dim V\cap\operatorname{spcl}(X)\geq \dim X > \dim V + \dim\operatorname{spcl}(X) - \dim\operatorname{spcl}(V)$, so the intersection $V\cap\operatorname{spcl}(X)$ is atypical. 
Specifically, $X$ is an atypical component of $V\cap\operatorname{spcl}(X)$ in $\operatorname{spcl}(V)$.
\end{remark}

\begin{defn}
We say that an algebraic variety $V\subseteq \mathbb{A}_{K}^{n}$ is \emph{Hodge-generic} if $V$ is not contained in any proper special subvariety of $\mathbb{A}_K^n$.    
\end{defn}

We will now recall the modular Zilber--Pink conjecture, which has appeared in the literature in many formulations.
Recall that $\mathscr{S}_n^k$ denotes the set of all special subvarieties of $\mathbb{A}^n_{\mathbb{C}}$ of dimension at most $k$.

\begin{conj}[Modular Zilber--Pink, unlikely version]
\label{conj:mzp}
For every positive integer $n$ and any Hodge-generic algebraic variety $V\subseteq \mathbb{A}_{\mathbb{C}}^{n}$, the set
\[\bigcup_{S\in\mathscr{S}_n^{n - (\dim V + 1)}}V\cap S\]
is not Zariski dense in $V$. 
\end{conj}

If $\dim V = n - 1$, then this follows from Theorem~\ref{thm:andre-oort}.

\begin{conj}[Modular Zilber--Pink, atypical version]
    Given a positive integer $n$, any algebraic variety $V\subseteq \mathbb{A}_{\mathbb{C}}^{n}$ contains only finitely many maximal atypical components. 
\end{conj}

\begin{conj}[Modular Zilber--Pink, optimal version]
    For every positive integer $n$ and any algebraic variety $V\subseteq \mathbb{A}_{\mathbb{C}}^{n}$, the set $\operatorname{Opt}(V)$ is finite.
\end{conj}

The equivalence between the atypical and optimal versions is given by \cite[Lemma 2.7]{habegger-pila:o-min}.
The equivalence between the atypical and unlikely versions can be found in \cite[\S12]{barroero-dill:distinguishedcategories}.
In fact, the various formulations of modular Zilber--Pink stated there are not only for subvarieties of $\mathrm{Y}(1)^n$, but also for subvarieties of a special subvariety of $\mathrm{Y}(1)^n$. 
But since any special subvariety admits a finite surjective projection to some $\mathrm{Y}(1)^k$, for the equivalence of the conjectures it suffices to state them as we have done here. 

For a Hodge-generic curve $V \subseteq \mathbb{A}^n$, the unlikely locus
\[\bigcup_{S\in\mathscr{S}_n^{n - 2}}V\cap S\]
is precisely equal to the union of the atypical components of $V$ and also equal to the union of the proper optimal subvarieties of $V$. 
Thus, for curves, the three versions of the conjecture are evidently equivalent.

The equivalences between the three formulations of modular Zilber--Pink are between the full conjectures, not variety to variety, so Theorem \ref{thm:genericmzp} does not readily yield that a variety $V\subseteq \mathbb{A}^n_{\mathbb{C}}$ satisfying the hypotheses of the theorem has only finitely many maximal atypical components, nor that $\operatorname{Opt}(V)$ is finite. 
But with the following results we obtain that neither the union of the atypical components of $V$, nor the union of the proper optimal subvarieties of $V$, can be Zariski dense in $V$, see Remark \ref{remark:atypical}.

\begin{lem}
    \label{lem:atypcomp}
    Let $V\subseteq\mathbb{A}^n_{K}$ be an algebraic variety, and let $\Upsilon$ denote the unlikely locus of $V$. 
    If $X$ is an atypical component of $V$, then $X\cap\Upsilon$ is Zariski dense in $X$. 
\end{lem}
\begin{proof}
First observe that any atypical component of $V$ of dimension zero belongs to $\Upsilon$.
Now suppose $X$ is a positive dimensional atypical component of $V$, say $\dim X=d_X>0$, and let $S$ be a special variety such that $X$ is a component of $V\cap S$ and $d_X > \dim V + \dim S -n$.
Let $1\leq i_{1}<\cdots<i_{d_X}\leq n$ be such that the corresponding coordinate projection $\mathrm{pr}_{\mathbf{i}}: \mathbb{A}^n_{K}\to \mathbb{A}^{d_X}_K$ restricts to a generically finite dominant map $\mathrm{pr}_{\mathbf{i}}\upharpoonright_{X}:X\to\mathbb{A}^{d_X}_{K}$, and let $U\subseteq\mathbb{A}^{d_X}_K$ be a Zariski dense open subset so that for every $\mathbf{u}\in U$, the fiber $\mathrm{pr}_{\mathbf{i}}\upharpoonright_{X}^{-1}(\mathbf{u})\subseteq X$ is finite. 
Then for every special point $\mathbf{s}\in U$, each point in the fiber $\mathrm{pr}_{\mathbf{i}}\upharpoonright_{X}^{-1}(\mathbf{s})$ is a zero-dimensional atypical component of $V$, indeed, each point $\mathbf{z}\in\mathrm{pr}_{\mathbf{i}}\upharpoonright_{X}^{-1}(\mathbf{s})$ is a component of $V\cap S\cap \mathrm{pr}_{\mathbf{i}}^{-1}(\mathbf{s})$ (note that the components of $S\cap \mathrm{pr}_{\mathbf{i}}^{-1}(\mathbf{s})$ are special varieties). 
Since special points are Zariski dense in $\mathbb{A}^{d_X}_K$, it follows that $U$ contains a Zariski dense set of special points, and since $\mathrm{pr}_{\mathbf{i}}\upharpoonright_{X}:X\to\mathbb{A}^{d_X}_{K}$ is a finite dominant map, this shows that the zero-dimensional atypical components of $V$ contained in $X$ are Zariski dense in $X$. 
\end{proof}

The following corollary was mentioned in the Introduction to help justify the hypotheses of Theorem \ref{thm:genericmzp}.

\begin{cor}
\label{cor:prunlikely}
Let $V\subseteq\mathbb{A}^n_K$ be an algebraic variety and let $\mathrm{pr}:\mathbb{A}^n_K\to\mathbb{A}^{\ell}_K$ denote some coordinate projection. 
Let $W$ be the Zariski closure of $\mathrm{pr}(V)$.
If the unlikely locus of $V$ is not Zariski dense in $V$, then the unlikely locus of $W$ is not Zariski dense in $W$. 
\end{cor}
\begin{proof}
Using the fibre-dimension theorem, let $U$ be a Zariski open dense subset of $W$ so that for every $\mathbf{u}\in U$, the irreducible components of the restricted fibre $\mathrm{pr}\upharpoonright_V^{-1}(\mathbf{u})$ have dimension $k\coloneqq\dim V-\dim W$. 
If the intersection between $U$ and the unlikely locus of $W$ is empty, then we immediately obtain that the unlikely locus of $W$ is not Zariski dense in $W$, so assume that there is $\mathbf{x}\in U$ which belongs to the unlikely locus of $W$. 
Then there is $S\in\mathscr{S}_{\ell}^{\ell-\dim W-1}$ such that $\mathbf{x}\in W\cap S$. 
Let $S' = \mathrm{pr}^{-1}(S) \cong S\times\mathrm{Y}(1)^m$, where $m=n-\ell$.
Then
\[\dim V + \dim S' = d + \dim S + n-\ell < d+n-\ell +\ell-\dim W = n+k,\]
which shows that each irreducible component of the fibre $\mathrm{pr}\upharpoonright_{V}^{-1}(\mathbf{x})$ is an atypical component of $V\cap S'$. 
By Lemma \ref{lem:atypcomp} we then know that this fibre is contained in the Zariski closure of the unlikely locus of $V$. 
Thus, if the unlikely locus of $W$ were Zariski dense in $W$, then there would be a Zariski dense set of fibres in $V$, each of whom is contained in the unlikely locus of $V$, and so the unlikely locus of $V$ would be Zariski dense in $V$. 
\end{proof}

\begin{remark}
\label{remark:atypical}
Let $V\subseteq\mathbb{A}^n_{\mathbb{C}}$ be an algebraic variety satisfying the hypotheses of Theorem \ref{thm:genericmzp}.
Lemma \ref{lem:atypcomp} shows that the union of atypical components is not Zariski dense in $V$ (and hence, the union of the proper optimal subvarieties of $V$ is also not Zariski dense in $V$, by Remark \ref{rem:maxatypisopt}). 
\end{remark}

\subsection{\texorpdfstring{$\Lambda$}{Lambda}-special varieties and modular Mordell--Lang}
\label{subsec:mordell-lang}
\begin{defn}
Given a subset $A\subseteq K$, the \emph{Hecke orbit} of $A$ is defined as
\[\mathrm{He}(A)\coloneqq\left\{z\in K : \exists a\in A\exists N\in \mathbb{N}(\Phi_N(z,a) = 0)\right\}.\]
\end{defn}

\begin{defn}
A subset $\Lambda\subseteq\mathbb{A}^n_K$ is called a \emph{structure of finite Hecke rank} if there is a set $\Lambda_0\subseteq K$ which contains only finitely many non-special points, and such that
\[\Lambda = \mathrm{He}\left(\Lambda_0\cup \mathscr{S}^0_1(K)\right)^n.\]
\end{defn}

Following \cite{aslanyan:atypical}, we also make the following definition.

\begin{defn}
Given a structure of finite Hecke rank $\Lambda\subseteq\mathbb{A}^n_K$, we say that a subvariety $S$ of $\mathbb{A}^n_K$ is $\Lambda$\emph{-special} if $S$ is weakly special and contains a point of $\Lambda$. 
In particular, a $\Lambda$-special point is just a point $\mathbf{v} \in \Lambda$.

We define the \emph{$\Lambda$-unlikely locus} of $V$ as the set of points $\mathbf{x} \in V$ such that there is a $\Lambda$-special variety $S$ with $\dim V+\dim S<n$ and $\mathbf{x}\in V\cap S$.
\end{defn}

Observe that a $\Lambda$-special subvariety of $\mathbb{A}_K^n$ without constant coordinates is always a strongly special subvariety of $\mathbb{A}_K^n$.

The main number-theoretic input we will use to prove Theorem \ref{thm:genericmzp} is the following result of Habegger and Pila. 

\begin{thm}[Modular Mordell--Lang, {{\cite{habegger-pila,pila:specialpts}}}]
\label{thm:modularmordell-lang}
Let $V\subseteq\mathbb{A}^n_{\mathbb{C}}$ be an algebraic subvariety and let $\Lambda\subseteq\mathbb{A}^n_{\mathbb{C}}$ be a structure of finite Hecke rank. 
Then $V$ contains only finitely many maximal $\Lambda$-special subvarieties. 
\end{thm}

This theorem is stated for complex varieties, but the same statement holds after replacing $\mathbb{C}$ with any algebraically closed field $K$ of characteristic zero. 
Indeed, since the statement involves countably many parameters (the set $\Lambda$ is countable and so is the field of definition of $V$), we can find a countable algebraically closed subfield $K_0\subseteq K$ over which $V$ is defined and such that $\Lambda\subseteq\mathbb{A}_{K_0}^n$. 
Since the statement of modular Mordell--Lang is geometric, we may then embed $K_0$ into $\mathbb{C}$, and since the result holds over $\mathbb{C}$, it will hold over $K_0$ and hence also over $K$. 

Theorem \ref{thm:modularmordell-lang} is closely related to another family of conjectures, usually referred to as the Andr\'e--Pink--Zannier conjectures (which are also special cases of the Zilber--Pink conjectures). 
The precise connection between Andr\'e--Oort, Andr\'e--Pink--Zannier and Mordell--Lang for Shimura varieties is given in \cite[Theorem 3.10]{aslanyan-daw}.
Various cases of Andr\'e--Pink--Zannier have been established thanks to the work of Richard and Yafaev \cite{richard-yafaev:generalizedAPZ}, see also \cite[\S1.2]{richard-yafaev:heightfunctions} for a history of the problem and the contributions from other important figures.

\subsection{Modular Ax--Schanuel}
\label{subsec:ax-schanuel}
The key ingredient in the proof of our main result is the differential formulation of the modular Ax--Schanuel theorem, first proven by Pila and Tsimerman \cite{pila-tsimerman:ax-schanuel}. 

It is well-known (e.g.~\cite[p.~20]{masser}) that $j$ satisfies the following algebraic differential equation (and, by a result of Mahler \cite{mahler}, none of lower order):
\begin{equation}
    \label{eq:j}
    0 = \frac{j'''}{j'} - \frac{3}{2}\left(\frac{j''}{j'}\right)^{2} + \frac{j^{2} -1968j + 2654208}{2j^{2}(j-1728)^{2}}\left(j'\right)^{2}.
\end{equation}
Define the rational function
\[\Psi\left(Y_0,Y_1,Y_2,Y_3\right):= \frac{Y_3}{Y_1} - \frac{3}{2}\left(\frac{Y_2}{Y_1}\right)^{2} + \frac{Y_0^{2} -1968Y_0 + 2654208}{2Y_0^{2}(Y_0-1728)^{2}}Y_1^{2}.\]
Then equation (\ref{eq:j}) can be rewritten as $\Psi(j,j',j'',j''') =0$.

\begin{thm}[{{\cite[Theorem 1.3]{pila-tsimerman:ax-schanuel}}}]
\label{thm:ax-schanuel}
    Let $(K,\Delta)$ be a differential field of characteristic zero, where $\Delta:=\{\partial_1,\ldots,\partial_m\}$ is a finite set of commuting derivations on $K$. 
    Let $C:=\bigcap_{i=1}^{m}\ker\partial_i$ denote the constant field of $K$. 
    Suppose $z_1,j_1,j_1',j_1'',j_1''',\ldots,z_n,j_n,j_n',j_n'',j_n'''\in K^\times$ satisfy that for every $i\in\{1,\ldots,n\}$ and every $k\in\{1,\ldots,m\}$
    \[\partial_k j_i = j_i'\partial_k z_i, \qquad \partial_k j_i' = j_i''\partial_k z_i, \quad \mbox{ and } \quad \partial_k j_i'' = j_i'''\partial_k z_i.\]
    Suppose further that for all $i\in\{1,\ldots,n\}$ we have $j_i\notin C$, $\Psi(j_i, j_i', j_i'', j_i''')=0$, and that for every distinct $i,k\in\{1,\ldots,n\}$ and every $N\in\mathbb{N}^+$ we have $\Phi_N(j_i,j_k)\neq 0$.
    Then 
    \[\mathrm{tr.deg.}_CC(z_1,j_1,j_1',j_1'',\ldots,z_n,j_n,j_n',j_n'')\geq 3n + \mathrm{rank}(\partial_{k}z_i)_{i,k}.\]
\end{thm}

\subsection{Some model theory of differential fields}
\label{subsec:dcf}

The proof of Theorem \ref{thm:genericmzp} will use some notions of differentially closed fields, in particular we will refer to the notion of \emph{generic} elements and \emph{Morley sequences}. 
In this section we recall definitions and properties we will need.
For background on this subject, we refer the reader to \cite{marker:modeltheorydcf}.

Throughout, let $(K,\partial)$ denote a differential field of characteristic zero.
This differential field is said to be \emph{differentially closed} if for every $n\in\mathbb{N}$ and every pair of polynomials $f\in K[X_0,\ldots,X_n]\setminus K[X_0,\ldots,X_{n-1}]$ and $g\in K[X_0,\ldots,X_{n-1}]\setminus\{0\}$ there is $a\in K$ such that 
\[f(a,\partial(a),\partial^2(a),\ldots,\partial^n(a))=0\quad\mbox{ and }\quad g(a,\partial(a),\ldots,\partial^{n-1}(a))\neq 0.\]
In particular, $K$ is an algebraically closed field. 
Furthermore, the set $\{x\in K : \partial(x) = 0\} = \ker\partial$, usually called the \emph{field of constants} of $K$, is also an algebraically closed field. 

Just like every field has an algebraic closure, every differential field $(K_0,\partial_0)$ can be embedded into a ``smallest'' differentially closed field $(K,\partial)$ called the \emph{differential closure} of $(K_0,\partial_0)$. 
In this case the field of constants $C_K=\ker\partial$ of $K$ is an algebraic extension of the field of constants $C_{K_0}=\ker\partial_0$ of $K_0$ \cite[Lemma 2.11]{marker:modeltheorydcf}. 

In analogy with the definition of the Zariski topology, given a positive integer $m$, we say that a subset $X\subseteq K^m$ is \emph{Kolchin closed} if $X$ is the solution set of a finite system of differential equations in $m$ variables. 
This defines the \emph{Kolchin topology} on $K^m$. 
By the Ritt--Raudenbush basis theorem \cite[Theorem 1.16]{marker:modeltheorydcf}, the Kolchin topology is Noetherian.

There is also an analogous version of Chevalley's theorem for constructible sets: the theory of differentially closed fields of characteristic zero exhibits quantifier elimination \cite[Theorem 2.4]{marker:modeltheorydcf}, which is reflected in the fact that any coordinate projection $K^m\to K^\ell$ (with $0<\ell\leq m$) of any Boolean combination of Kolchin closed subsets of $K^m$ is a Boolean combination of Kolchin closed subsets of $K^\ell$.

Now we turn our attention to the notion of generic elements of irreducible subsets of $K^m$. 
When working with an irreducible affine complex algebraic variety $W\subseteq\mathbb{C}^m$, using Notherianity of the Zariski topology we can find an algebraically closed field $E\subseteq\mathbb{C}$ such that $W$ is defined over $E$ and $\mathrm{tr.deg.}_{\mathbb{Q}}E$ is finite.
We say that a point $\mathbf{w}\in W$ is \emph{generic in $W$ over $E$} if 
\[\mathrm{tr.deg.}_EE(\mathbf{w}) = \dim W.\]
This notion of genericity can be rephrased in the following way: $\mathbf{w}$ is an element of $W$ which is not contained in any proper Zariski closed subset of $W$ which is also defined over $E$. 
The fact that generic elements always exist with respect to a field $E$ as above is a consequence of the presence of uncountably many algebraically independent elements in $\mathbb{C}$; in model theoretic terms we say that $\mathbb{C}$ is a \emph{saturated} model of the theory of algebraically closed fields of characteristic zero.
Because of this, we can produce a sequence $\{\mathbf{w}_i\}_{i\in\mathbb{N}}$ such that for every $i\in\mathbb{N}$, $\mathbf{w}_{i}$ is generic in $W$ over $\overline{E(\mathbf{w}_0,\ldots,\mathbf{w}_{i-1})}$. 
Equivalently, for every $n\in\mathbb{N}$ the tuple $(\mathbf{w}_0,\ldots,\mathbf{w}_n)$ is generic in $W^{n+1}$ over $E$. 
Such a sequence of generic elements is called a \emph{Morley sequence in $W$ over $E$} (for the theory of algebraically closed fields).

In the context of differentially closed fields, given an irreducible Kolchin closed set $V\subseteq K^m$ we first take a differential subfield $F\subseteq K$ so that $V$ is defined by differential equations with coefficients in $F$, and we say that an element $\mathbf{v}\in V$ is \emph{generic in $V$ over $F$} (with respect to the Kolchin topology) if $\mathbf{v}$ is not contained in any proper Kolchin closed subset of $V$ defined over $F$. 
By Notherianity of the Kolchin topology, we may assume that $F$ is finitely generated (as a differential field) over $C=\ker\partial$.
A \emph{Morley sequence in $V$ over $F$} is a sequence $\{\mathbf{v}_i\}_{i\in\mathbb{N}}$ of elements of $V$ such that for every $n\in\mathbb{N}$, the tuple $(\mathbf{v}_0,\ldots,\mathbf{v}_n)$ is generic in $V^{n+1}$ over $F$.
The existence of generic elements (and hence of Morley sequences) is guaranteed if we further assume that $(K,\partial)$ is an $\omega$-saturated model of the theory of differentially closed fields (it is a standard result of model theory that every model of any theory can be embedded into a sufficiently saturated model of that theory). 

Furthermore, the theory of differentially closed fields of characteristic zero also has the property that it is \emph{$\omega$-stable} \cite[Lemma 2.8]{marker:modeltheorydcf}.
What this will mean for us is that any differential field of characteristic zero can be embedded into a saturated differentially closed field, see \cite[Theorem 5.2.6 and Lemma 5.2.9]{tent-ziegler}.

\section{Main Results} 
\label{sec:main}

\subsection{Partitions}
\label{subsec:partitions}
Let $K$ denote an algebraically closed field of characteristic zero.
In this section we explain how we will use partitions of an integer $n$ to encode the dimension of weakly special subvarieties of $\mathbb{A}^n_{K}$. 
\begin{defn}
Let $n$ denote a positive integer.
A \emph{partition of $n$} is a set $\mathscr{P}$ of non-empty subsets of $\{1,\ldots,n\}$ satisfying the following conditions
\begin{enumerate}[(a)]
    \item $\bigcup\mathscr{P} = \{1,\ldots,n\}$, and
    \item for every $P_1,P_2\in\mathscr{P}$, $P_1\neq P_2\implies P_1\cap P_2=\emptyset$.
\end{enumerate} 
\end{defn}

\begin{defn} 
Let $\mathscr{P}$ be a partition  of $n$ and let $\mathscr{P}_0$ denote a subset of $\mathscr{P}$ (possibly empty).
Given a weakly special subvariety $S\subseteq\mathbb{A}^n_K$, we say that $S$ has \emph{special type} $(\mathscr{P},\mathscr{P}_0)$ if the following conditions are met:
\begin{enumerate}[(a)]
    \item For all $i,k\in\{1,\ldots,n\}$ distinct, the coordinates $X_i$ and $X_k$ are unfixed and modularly dependent on $S$ if and only if there is $P\in\mathscr{P}\setminus\mathscr{P}_0$ such that $i,k\in P$.
    
    \item For all $i\in\{1,\ldots,n\}$, the coordinate $X_i$ has a fixed value on $S$ if and only if there is $P\in\mathscr{P}_0$ such that $i\in P$. 

    \item For all $i,k\in\bigcup\mathscr{P}_0$, the values of the coordinates $X_i$ and $X_k$ on $S$ are modularly dependent if and only if there is $P\in\mathscr{P}_0$ such that $i,k\in P$.
\end{enumerate}
We denote that $S$ has special type $\mathscr{P}$ by $S\sim_{\rm s}(\mathscr{P},\mathscr{P}_0)$.
Note that a weakly special variety does not determine a unique special type; however the dimension of $S$ is determined by its special type
\[\dim S = |\mathscr{P}| - |\mathscr{P}_0|.\]
Indeed, the quantity $|\mathscr{P}| - |\mathscr{P}_0|$ counts the number of modularly independent variables on $S$ whose value is not fixed on $S$. 
\end{defn}

We remark that if $S\sim_{\rm s}(\mathscr{P},\mathscr{P}_0)$ and $i\in\{1,\ldots,n\}$ is such that $\{i\}\in\mathscr{P}$, then the coordinate $X_i$ either has fixed value on $S$ (if $\{i\}\in\mathscr{P}_0$) or is modularly independent from all other coordinates. 

\begin{defn}
    Let $(K,\partial)$ denote a differential field of characteristic zero and let $C$ denote the field of constants, that is $C=\ker\partial$. 
    Consider the action of $\mathrm{GL}_2(C)$ on $\mathbb{P}^1_K=K\cup\{\infty\}$ given by:
    \[\begin{pmatrix}
        a & b\\
        c& d
    \end{pmatrix}z\coloneqq\frac{az+b}{cz+d}.\]
    Given a partition $\mathscr{P}$ of $n$ and a tuple $\mathbf{x}\in \mathbb{A}_K^n$, we will say  that $\mathbf{x}$ is of \emph{M\"obius type} $\mathscr{P}$, denoted $\mathbf{x}\sim_{\rm m}\mathscr{P}$, if for every $P\in\mathscr{P}$ and every $i,k\in P$, the elements $x_i$ and $x_k$ are in the same $\mathrm{GL}_2(C)$-orbit.
    We remark that, given a partition $\mathscr{P}$ of $n$, the set of elements of $\mathbb{A}_K^n$ of M\"obius type $\mathscr{P}$ is definable in the language of differential fields. 
\end{defn}

\subsection{The unlikely locus in differential fields}
\label{subsec:diffunlikely}
One of the usual obstacles when working on problems of unlikely intersections is that the collection of special subvarieties of $\mathbb{A}^n_\mathbb{C}$ does not form part of an algebraic family of varieties, the problem being that modular polynomials have unbounded degree. 

A standard way of circumventing this issue is by using the $j$-function to move the problem to the upper half-plane $\mathbb{H}$, as modular relations in $\mathbb{A}^n_{\mathbb{C}}$ can be interpreted as $\mathrm{GL}_2(\mathbb{Q})^+$ relations on $\mathbb{H}^n$. 
Since $\mathrm{GL}_2(\mathbb{Q})^+\subseteq\mathrm{GL}_2(\mathbb{R})\subseteq\mathrm{GL}_2(\mathbb{C})$, one can use the actions from these algebraic groups to produce convenient families to work with.

We now explain how to build convenient families in the context of differentially closed fields. 
Consider a differentially closed field $(K,\partial)$ of characteristic zero with field of constants $C \coloneqq \ker\partial$, and let $V\subseteq\mathbb{A}^n_K$ be an algebraic variety.
Let $F$ denote an algebraically closed subfield of $K$ satisfying that $V$ is defined over $F$, $C\subseteq F$ and $\mathrm{tr.deg.}_CF$ is finite (such an $F$ always exists). 

We remark that since $K$ is differentially closed, then $C$ is algebraically closed, and so every special variety (including special points) is defined over $C$. 

Recall that the unlikely locus of $V$ (which we now denote $\Upsilon(V)$) was defined in \S\ref{subsec:mzp} as
\[\Upsilon(V) \coloneqq \bigcup_{S\in\mathscr{S}_n^{n-\dim V-1}} (V\cap S).\]
We will be using the modular Mordell--Lang theorem in order to prove Theorem \ref{thm:genericmzp}, so we will also need to consider the intersection of $V$ with $\Lambda$-special varieties. 
Let $\Lambda_0\subseteq F$ be a finite modularly independent subset of $K$ which does not include special points and set $\Lambda := \mathrm{He}(\Lambda_0\cup \mathscr{S}_1^0(K))^n$. 
Given $k\leq n$, we denote by $\mathscr{S}^k_{\Lambda}$ the set of all $\Lambda$-special subvarieties $S$ of $\mathbb{A}^n_K$ satisfying: $\dim S\leq k$.
So we can denote the $\Lambda$-unlikely locus of $V$ as
\[\Upsilon_\Lambda(V) \coloneqq \bigcup_{S\in\mathscr{S}_\Lambda^{n-\dim V-1}} (V\cap S).\]

For every $S\in\mathscr{S}_\Lambda^{n-\dim V-1}$ there is a partition $\mathscr{P}$ of $n$ and a subset $\mathscr{P}_0\subseteq\mathscr{P}$ such that $S\sim_{\rm s}(\mathscr{P},\mathscr{P}_0)$. 
Furthermore, since we are working in a differential field, any point $\mathbf{x}\in V\cap S$ can be separated as $\mathbf{x}=(\mathbf{a},\mathbf{c})$, where the entries of $\mathbf{c}$ are all in $C$ and no entry of $\mathbf{a}$ is in $C$. 
We can then see $\Upsilon_\Lambda(V)$ as the union of sets of the form
\[\Upsilon_{\Lambda,(\mathscr{P},\mathscr{P}_0),\nu}(V)\coloneqq\bigcup_{S\in\mathscr{S}^{n-\dim V-1}_{\Lambda}}\left\{(\mathbf{a},\mathbf{c})\in(K\setminus C)^{n-\nu}\times C^{\nu} : \left[(\mathbf{a},\mathbf{c})\in V\cap S\right] \wedge \left[S\sim_{\rm s}(\mathscr{P},\mathscr{P}_0)\right]\right\}.\]
Here, and in everything that follows, we will implicitly assume that the coordinates have been permuted so that the entries with values in $C$ are always the last $\nu$ ones.
This is a harmless assumption since, on the one hand, there are only finitely many possible coordinate permutations, and, on the other hand, a finite union of sets which are not Zariski dense is also not Zariski dense.

The set $\Upsilon_{\Lambda,(\mathscr{P},\mathscr{P}_0),\nu}(V)$ is still not useful as we are still taking unions over special varieties of unbounded degree, so in view of the comment made earlier, we need to ``move'' the problem to $\mathbb{H}$.
The way to do that in a differentially closed field (where there is no analytic topology, and hence, no notion of upper half-plane) is to use the differential equation (\ref{eq:j}).
Consider the following function on pairs of elements of $K\setminus C$:
    \[\psi(x,y)\coloneqq\Psi\left(y, \partial_x y, \partial_x^2 y, \partial_x^3 y\right),\]
    where $\partial_x\coloneqq \frac{\partial}{\partial x}$ and $\Psi$ is the rational function defined in \S\ref{subsec:ax-schanuel}.
    Now we consider the set 
    \[\Xi\coloneqq\{(x,y)\in(K\setminus C)^2 :\psi(x,y)=0\}.\]
    A point $(x,y)\in\Xi$ can be informally thought of as $y$ representing the $j$-image of $x$ (a more accurate interpretation of the points of $\Xi$ can be framed in terms of ``blurrings'' as defined in \cite{aslanyan-kirby}).
    For every $y_0\in K\setminus C$ the fibre in $\Xi$ above $y_0$ is non-empty (because $K$ is differentially closed) and consists of a whole $\mathrm{GL}_2(C)$-orbit (see e.g.~\cite[Lemma 4.9]{aslanyan:adequate}), whereas by \cite[Theorem 1.1]{aslanyan:strongminj} and \cite{freitag-scanlon:strongminj} for each $x_0\in K\setminus C$ the fibre
    \[\{y\in K\setminus C : \psi(x_0,y)=0\}\subseteq\Xi\]
    is strongly minimal and has trivial pregeometry.

Using that $K$ is differentially closed, for each $\lambda\in\Lambda_0\setminus C$ we can find $\tau_\lambda\in K\setminus C$ so that $(\tau_\lambda,\lambda)\in\Xi$. 
We choose a finite set $\{\tau_\lambda\}_{\lambda\in\Lambda_0\setminus C}$ and we assume $\{\tau_\lambda\}_{\lambda\in\Lambda_0\setminus C}\subseteq F$.

Let $S$ denote a $\Lambda$-special variety with $S\sim_{\rm s}(\mathscr{P},\mathscr{P}_0)$.
Then for every $i\in P_0\coloneqq\bigcup\mathscr{P}_0$, the coordinate $X_i$ will have a fixed value on $S$, but this value need not be in $C$, instead it can be a value in the Hecke orbit of some element of $\Lambda_0\setminus C$.
This distinction will be significant in future computations, so we let $P_0'\subseteq P_0$ denote a subset which will be thought of as those coordinates which have fixed value on $S$ and the value is in $\mathrm{He}(\Lambda_0)\setminus C$. 
We now define
    \begin{align*}
        &\Omega_{(\mathscr{P},\mathscr{P}_0),\nu,P_0'}(V)\coloneqq\left\{(\mathbf{a},\mathbf{c})\in(K\setminus C)^{n-\nu}\times C^{\nu} :
        (\mathbf{a},\mathbf{c})\in V \wedge\exists\boldsymbol{\tau}\in K^{n} \right.\\ 
        &\left.\left[ 
\bigwedge_{i=1}^{n-\nu}(\tau_i, a_i)\in\Xi \wedge \bigwedge_{i=n-\nu+1}^n\tau_i\in C \wedge  \boldsymbol{\tau}\sim_{\rm m} \mathscr{P}\wedge \bigwedge_{i\in P_0'}\exists g\in\mathrm{GL_2}(C)\bigvee_{\lambda\in\Lambda_0\setminus C}(g\tau_i,\lambda)\in\Xi\right]\right\}.
    \end{align*}
This definition we gave shows that $\Omega_{(\mathscr{P},\mathscr{P}_0),\nu,P_0'}(V)$ is definable in the language of differential fields, with parameters in $F$. 

\begin{defn}
Given a subset $P_0\subseteq\{1,\ldots,n\}$ and $\nu\in\{0,\ldots,n\}$, let $\mathscr{A}_\nu(P_0)$ be the set of subsets $P_0'$ of $P_0$ satisfying $P_0\setminus P_0'\subseteq\{n-\nu+1,\ldots,n\}$.
As usual, given a special type $(\mathscr{P},\mathscr{P}_0)$, we set $P_0=\bigcup\mathscr{P}_0$. 
The \emph{differential $\Lambda$-unlikely locus of $V$}, denoted $\Omega_\Lambda(V)$, is the union 
\[\Omega_\Lambda(V)\coloneqq\bigcup_{\substack{(\mathscr{P},\mathscr{P}_0) \\ |\mathscr{P}|-|\mathscr{P}_0| < n-\dim V}}\bigcup_{\nu=0}^{n}\,\bigcup_{P_0'\in\mathscr{A}_\nu(P_0)}\Omega_{(\mathscr{P},\mathscr{P}_0),\nu,P_0'}(V).\]
\end{defn}

\begin{lem}
    \label{lem:upsiloninomega}
    With the notation above
    \[\Upsilon_{\Lambda,(\mathscr{P},\mathscr{P}_0),\nu}(V)\subseteq\bigcup_{P_0'\in \mathscr{A}_\nu(P_0)}\Omega_{(\mathscr{P},\mathscr{P}_0),\nu,P_0'}(V).\]
\end{lem}
\begin{proof}
    Suppose $(\mathbf{a},\mathbf{c})\in\Upsilon_{\Lambda,(\mathscr{P},\mathscr{P}_0),\nu}(V)$, then by definition there is a $\Lambda$-special variety $S$ such that $(\mathbf{a},\mathbf{c})\in V\cap S$ and $S\sim_{\rm s}(\mathscr{P},\mathscr{P}_0)$. 
    Since $K$ is differentially closed, there is $\boldsymbol{\tau}\in (K\setminus C)^{n-\nu}\times C^\nu$ such that for every $i\in\{1,\ldots,n-\nu\}$ we have $(\tau_i,a_i)\in\Xi$.
    Choose $P\in\mathscr{P}$ and let $i,k\in P$.
    If $P\notin\mathscr{P}_0$, then the coordinates $X_i$ and $X_k$ are modularly dependent on $S$, so the corresponding entries of $(\mathbf{a},\mathbf{c})$ are also modularly dependent (this includes the case $i=k$), hence we have either $i,k\in\{1,\ldots,n-\nu\}$ or $i,k\in\{n-\nu+1,\ldots,n\}$ (values in $C$ cannot be modularly dependent with values in $K\setminus C$), and thus there is $g\in\mathrm{GL}_2(C)$ such that $g\tau_i = \tau_k$, either because $i,k\in\{1,\ldots,n-\nu\}$ and $a_i,a_k$ are modularly dependent and we apply \cite[Lemmas 4.9 and 4.10]{aslanyan:adequate}, or because  $i,k\in\{n-\nu+1,\ldots,n\}$ and so $\tau_i,\tau_k\in C$.
    
    On the other hand, if $P\in\mathscr{P}_0$, then the coordinates $X_i$ and $X_k$ have fixed value on $S$ and these values are modularly dependent, so they are either both special points or both elements of $\operatorname{He}(\Lambda_0)$.
    If these values are special points, then they are in $C$ and so $i,k\in\{n-\nu+1, \ldots,n\}$, which shows that $\tau_i,\tau_k\in C$ and hence there is $g\in\mathrm{GL}_2(C)$ such that $g\tau_i=\tau_k$.
    The same thing happens if these values are both in $\operatorname{He}(\Lambda_0)\cap C$.
    Again because values in $C$ cannot be modularly dependent with values in $K\setminus C$, the last case that remains is when these values are both in $\operatorname{He}(\Lambda_0)\setminus C$ and modularly dependent, then by \cite[Lemmas 4.9 and 4.10]{aslanyan:adequate} as above, there is $g\in\mathrm{GL}_2(C)$ such that $g\tau_i=\tau_k$. 
    Thus $\boldsymbol{\tau}\sim_{\rm m}\mathscr{P}$ and if we let $P_0'$ denote the subset of $P_0$ of coordinates having a fixed value on $S$ and that value is in $\operatorname{He}(\Lambda_0)\setminus C$, we get $P_0'\in\mathscr{A}_\nu(P_0)$ and
    \[(\mathbf{a},\mathbf{c})\in \Omega_{(\mathscr{P},\mathscr{P}_0),\nu,P_0'}(V).\qedhere\] 
\end{proof}
    
Going back to the original unlikely locus of $V$, when $\Lambda_0=\emptyset$ we may simplify the above definitions to only consider the case $P_0'=\emptyset$, and so we define the \emph{differential version of the unlikely locus of $V$} as
\[\Omega(V)\coloneqq\bigcup_{\substack{(\mathscr{P},\mathscr{P}_0) \\ |\mathscr{P}|-|\mathscr{P}_0| < n-\dim V}}\bigcup_{\nu=|P_0| }^{n}\Omega_{(\mathscr{P},\mathscr{P}_0),\nu,\emptyset}(V).\]
Finally, we remark that if $\boldsymbol{\tau}\in K^n$ is a tuple witnessing the presence of $(\mathbf{a},\mathbf{c})$ in $\Omega_{(\mathscr{P},\mathscr{P}_0),\nu,P_0'}(V)$, then (recall that we are assuming $\{\tau_\lambda\}_{\lambda\in\Lambda_0\setminus C}\subseteq F$)
    \begin{equation}
        \label{eq:tauupperbound}
        \mathrm{tr.deg.}_FF(\boldsymbol{\tau})\leq|\mathscr{P}| - |\mathscr{P}_0|,
    \end{equation}
    because every coordinate of $\boldsymbol{\tau}$ which belongs to $P_0$ is in the $\mathrm{GL}_2(C)$-orbit of some element of $F$.

\subsection{Proving  the main theorem}
\label{subsec:proofmainthm}
The main step in proving Theorem \ref{thm:genericmzp} is given by the following theorem. 
As we will see later, modular Mordell--Lang allows us to reduce the proof of Theorem \ref{thm:genericmzp} to the case where no coordinate has a fixed value on $V$, but at the cost of showing non-Zariski density of $\Upsilon_\Lambda(V)$, for any structure $\Lambda$ of finite Hecke rank.

\begin{thm}
\label{thm:mainj}
    Let $V\subseteq\mathbb{A}^n_{\mathbb{C}}$ be an algebraic variety where no coordinate is fixed in value, and let $\Lambda\subseteq\mathbb{A}^n_{\mathbb{C}}$ be a structure of finite Hecke rank. 
    Let $d = \dim V$. 
    Suppose that for every coordinate projection $\mathrm{pr}:\mathbb{A}^n_{\mathbb{C}}\to\mathbb{A}^{d+2}_{\mathbb{C}}$, $\mathrm{pr}(V)$ is not definable over $\overline{\mathbb{Q}}$. 
    If $V$ is Hodge-generic, then the $\Lambda$-unlikely locus of $V$ is not Zariski dense in $V$. 
\end{thm}
\begin{proof}
If $d=n$ then $\Upsilon_\Lambda(V)=\emptyset$ and there is nothing to prove. 
If $d = n-1$, then $\Upsilon_\Lambda(V)$ contains only $\Lambda$-special points.
But if such points were Zariski dense in $V$, then by modular Mordell--Lang (Theorem \ref{thm:modularmordell-lang}), $V$ would be a $\Lambda$-special variety. 
Since no coordinate is fixed on $V$, then this implies that $V$ is a proper strongly special variety, showing that $V$ is not Hodge-generic.  
So from now on, we assume that $d\leq n-2$. 

Let $F_1\subset\mathbb{C}$ be an algebraically closed subfield such that $\operatorname{tr.deg.}_{\overline{\mathbb{Q}}}F_1$ is finite and $V$ is defined over $F_1$. 
Let $\Lambda_0\subset\mathbb{C}$ be a finite modularly independent subset which does not include singular moduli and such that $\Lambda = \mathrm{He}(\Lambda_0\cup \mathscr{S}_1^0(\mathbb{C}))^n$. 
We assume that $\Lambda_0\subseteq F_1$.

We now embed $F_1$ into a saturated differentially closed field $(K,\partial)$ of characteristic zero so that $F_1\subseteq F\subseteq K$, where $F$ is an algebraically closed field containing the field of constants $C$ with $\mathrm{tr.deg.}_CF$ finite, and for every coordinate projection $\mathrm{pr}:\mathbb{A}_K^n\to\mathbb{A}^{d+2}_K$, the image $\mathrm{pr}(V)$ is not definable over $C$ (here we have abused notation slightly, $V$ was being identified with its $\mathbb{C}$-points in the previous paragraph, but for the rest of this proof $V$ will be identified with its $K$-points).
For this, let $\{t_1,\ldots,t_r\}$ denote a transcendence basis for $F_1$ over $\overline{\mathbb{Q}}$, and choose $t_{r+1}\in\mathbb{C}\setminus F_1$. 
Then for each $i\in\{1,\ldots,r\}$ there is a derivation $\partial_i:F_1(t_{r+1})\to F_1(t_{r+1})$ such that for every $k\in\{1,\ldots,r+1\}$
\[\partial_i(t_k)\coloneqq\left\{\begin{array}{cc}
    1 & i=k \\
    0 & i\neq k
\end{array}\right..\]
Define $\partial:F_1(t_{r+1})\to F_1(t_{r+1})$ as $\partial := \sum_{i=1}^{r}t_{r+1}^i\partial_i$.
In this way, if $x\in F_1$, then $\partial_i(x)\in F_1$ for all $i\in\{1,\ldots,r\}$, and so if additionally $\partial(x)=0$, then $x\in \overline{\mathbb{Q}}$. 
We now embed $(F_1(t_{r+1}),\partial)$ into a saturated differentially closed field $(K,\partial)$ and let $C=\{x\in K:\partial(x)=0\}$. 
Since the derivation on $K$ extends the derivation on $F_1(t_{r+1})$, we get that $\overline{F_1}\cap C = \overline{\mathbb{Q}}$ (where by abuse of notation $\overline{\mathbb{Q}}$ is meant to represent the algebraic closure of $\mathbb{Q}$ in $K$). 
Therefore, since for every coordinate projection $\mathrm{pr}:\mathbb{A}_K^n\to\mathbb{A}^{d+2}_K$, the image $\mathrm{pr}(V)$ is definable over $\overline{F_1}$, it cannot also be definable over $C$ (this is a standard consequence of the minimality of the field of definition).\footnote{To see this, the model theory oriented reader may observe that $\mathrm{pr}(V)$ is a definable set, its field of definition is generated by a canonical parameter, and the theory of algebraically closed fields of characteristic zero eliminates imaginaries, while the algebraic geometry inclined reader may replace $\mathrm{pr}(V)$ with its Zariski closure everywhere in the proof and in the statement of the theorem.}
Set $F$ to be the algebraic closure of the compositum of $\overline{F_1}$ and $C$ in $K$, thus $\operatorname{tr.deg.}_CF$ is finite. 

Since $K$ is differentially closed, for each $\lambda\in\Lambda_0\setminus C$ we can choose $\tau_\lambda\in K\setminus C$ so that $(\tau_\lambda,\lambda)\in \Xi$. 
By increasing $F$ if necessary (but keeping $\operatorname{tr.deg.}_CF$  finite), we may assume that $\tau_\lambda\in F$ for each $\lambda\in\Lambda_0\setminus C$. 
    
To prove the theorem we will assume that $V$ is Hodge-generic and some $\Omega_{(\mathscr{P},\mathscr{P}_0),\nu,P_0'}(V)$ is Zariski dense in $V$, and we will show that there is a proper Zariski closed subset $Z\subsetneq V$ such that the intersection of $V$ with any $\Lambda$-special subvariety of special type $(\mathscr{P},\mathscr{P}_0)$ is either likely or contained in $Z$. 
By Lemma \ref{lem:upsiloninomega}, and since there are only finitely many valid choices for the data $(\mathscr{P},\mathscr{P}_0),\nu,P_0'$ for fixed $n$, this will suffice to prove the theorem. 
Indeed, we may choose $K$ to have cardinality at most continuum, thus $K$ embeds as a field in $\mathbb{C}$.

We want to consider generic elements of $\Omega_{(\mathscr{P},\mathscr{P}_0),\nu,P_0'}(V)$ (as described in \S\ref{subsec:dcf}). 
To do this, we first need to ensure that $\Omega_{(\mathscr{P},\mathscr{P}_0),\nu,P_0'}(V)$ is irreducible (in the Kolchin topology). 
If $\Omega_{(\mathscr{P},\mathscr{P}_0),\nu,P_0'}(V)$ is not irreducible, then it has only finitely many irreducible components (by Noetherianity), so at least one of those components will also be Zariski dense in $V$. 
By quantifier elimination of the theory of differentially closed fields (of characteristic zero), the irreducible components are also definable, and so we may work with an appropriate irreducible component of $\Omega_{(\mathscr{P},\mathscr{P}_0),\nu,P_0'}(V)$ instead, if necessary. 

Let $\{(\mathbf{a}_i,\mathbf{c}_i)\}_{i\in\mathbb{N}}$ denote a Morley sequence in $\Omega_{(\mathscr{P},\mathscr{P}_0),\nu,P_0'}(V)$ over $F$.\footnote{We remark that since $\operatorname{tr.deg.}_CF$ is finite, a type over $F$ is equivalent to a type over a transcendence basis of $F$ over $C$, hence the saturation of $K$ suffices to guarantee the existence of this Morley sequence even though $K$ and $F$ have the same cardinality.}
For each $i\in\mathbb{N}$, choose $\boldsymbol{\tau}_i\in K^{n}$ witnessing $(\mathbf{a}_i,\mathbf{c}_i)\in\Omega_{(\mathscr{P},\mathscr{P}_0),\nu,P_0'}(V)$.
We first observe that for every $i\in\mathbb{N}$, the coordinates of $\mathbf{a}_i$ are modularly independent because $V$ is Hodge-generic and no coordinate is fixed on $V$, so we can always choose a realization of $\Omega_{(\mathscr{P},\mathscr{P}_0),\nu,P_0'}(V)$ which avoids any finite collection of proper $\Lambda$-special varieties, and then by saturation of $K$, we may choose $\mathbf{a}_i$ avoiding all $\Lambda$-special varieties.
Furthermore, since $\{(\mathbf{a}_i,\mathbf{c}_i)\}_{i\in\mathbb{N}}$ is a Morley sequence, for every distinct $i,k\in\mathbb{N}$, the entries of $\mathbf{a}_i$ are modularly independent from the entries of $\mathbf{a}_k$.

The proof is separated into various cases, all of which involve similar manipulations of inequalities with transcendence degree. 
The main difference between each of the cases is the upper bounds we use for $\mathrm{tr.deg.}_{F}F(\boldsymbol{\tau}_i)$ and $\mathrm{tr.deg.}_{F(\boldsymbol{\tau}_i)}F(\boldsymbol{\tau}_i,(\mathbf{a}_i,\mathbf{c}_i))$.

We first deal with the case $\nu = 0$.
By Ax--Schanuel (Theorem \ref{thm:ax-schanuel}) we have that for every positive integer $m$
\[\mathrm{tr.deg.}_CC(\boldsymbol{\tau}_1,\ldots,\boldsymbol{\tau}_m,\mathbf{a}_1,\ldots,\mathbf{a}_m)\geq nm+1.\]
On the other hand, $(\mathbf{a}_1,\ldots,\mathbf{a}_m)\in V^m$ for each $m\in\mathbb{N}$, so combining this with (\ref{eq:tauupperbound}) gives
\begin{align*}
\mathrm{tr.deg.}_CF(\boldsymbol{\tau}_1,\ldots,\boldsymbol{\tau}_m,\mathbf{a}_1,\ldots,\mathbf{a}_m)&\leq \mathrm{tr.deg.}_CF + \mathrm{tr.deg.}_{F}F(\boldsymbol{\tau}_1,\ldots,\boldsymbol{\tau}_m) + md\\
    &\leq \mathrm{tr.deg.}_CF + m(|\mathscr{P}|-|\mathscr{P}_0|) + md.
\end{align*}
Since these two inequalities must hold for all values of $m$, we conclude that for every $\Lambda$-special variety $S$ with $S\sim_{\rm s}(\mathscr{P},\mathscr{P}_0)$, we have
\[n \leq |\mathscr{P}|-|\mathscr{P}_0|+d = \dim S + \dim V,\]
which shows that $V$ and $S$ are likely to intersect. 

From now on we assume that $\nu>0$. 
Considering all coordinate projections $\mathrm{pr}:\mathbb{A}^n_K\to\mathbb{A}_K^{d+2}$ and using that $\mathrm{pr}(V)$ is not definable over $C$, we see that the set of points of $V$ for which at least $d+2$ many entries are in $C$ is also not Zariski dense in $V$. 
Therefore we assume $\nu\leq d+1 \leq n-1$. 

The condition that $\mathrm{pr}(V)$ is not definable over $C$ for any coordinate projection $\mathrm{pr}:\mathbb{A}_K^n\to\mathbb{A}_K^{d+2}$ has the following consequence: any $d+2-\nu$ entries of every $\mathbf{a}_i$ must be algebraically dependent over $F$. 
Indeed, otherwise we could find a subtuple $\mathbf{b}$ of $\mathbf{a}_i$ of length $d+2-\nu$ whose entries are algebraically independent over $F$, and then by considering the coordinate projection which maps $(\mathbf{a}_i,\mathbf{c}_i)\mapsto(\mathbf{b},\mathbf{c}_i)$, we would get a coordinate projection $\mathbb{A}_K^n\to\mathbb{A}_K^{d+2}$ where the image of $V$ is defined over $F$, but where no polynomial relations defining the image of $V$ involve the coordinates corresponding to $\mathbf{b}$, only those in $\mathbf{c}$.
Since $\Omega_{(\mathscr{P},\mathscr{P}_0),\nu,P_0'}(V)$ is Zariski dense in $V$, this would then imply that $\mathrm{pr}(V)$ is defined over $C$, contradicting the hypotheses.
 We conclude from this that 
 \begin{equation}
     \label{eq:uppertda_i}
     \mathrm{tr.deg.}_{F}F(\mathbf{a}_i,\mathbf{c}_i) = \mathrm{tr.deg.}_{F}F(\mathbf{a}_i)\leq d+1-\nu=\dim V+1-\nu.
 \end{equation}

We now distinguish two cases. 
Recall that $P_0=\bigcup\mathscr{P}_0$ is the set of coordinates that have a fixed value on a weakly special variety $S\sim_{\rm s}(\mathscr{P},\mathscr{P}_0)$. 
When $S$ is $\Lambda$-special, the values of the coordinates in $P_0$ need not be in $C$, and we use $P_0'$ to denote the subset of $P_0$ of coordinates whose value is not in $C$. 
So then the number $\ell_1\coloneqq|P_0| - |P_0'|$ is the number of coordinates that have a fixed value on $S$ and this value is in $C$.
By Lemma \ref{lem:upsiloninomega} we assume $P_0\setminus P_0'\subseteq\{n-\nu+1,\ldots,n\}$.

Suppose first that $\ell_1 < \nu$. 
This means that there is $k\in\{n-\nu+1,\ldots,n\}$ such that $k\notin P_0$ and for every $i\in\mathbb{N}$, the $k$-th coordinate of $\boldsymbol{\tau}_i$ is in $C$.
Hence we can improve the upper bound of (\ref{eq:tauupperbound}) to $\mathrm{tr.deg.}_{F}F(\boldsymbol{\tau}_i)\leq |\mathscr{P}|-|\mathscr{P}_0|-1$ for all $i\in\mathbb{N}$.
  Combining this with (\ref{eq:uppertda_i}) and using Ax--Schanuel (Theorem \ref{thm:ax-schanuel}) gives
\begin{align*}
    m(n-\nu) +1&\leq \mathrm{tr.deg.}_CF(\boldsymbol{\tau}_1,\ldots,\boldsymbol{\tau}_m,\mathbf{a}_1,\mathbf{c}_1\ldots,\mathbf{a}_m,\mathbf{c}_m)\\
    &\leq  \mathrm{tr.deg.}_CF +\mathrm{tr.deg.}_{F}F(\boldsymbol{\tau}_1,\ldots,\boldsymbol{\tau}_m) +\mathrm{tr.deg.}_{F}F(\mathbf{a}_1,\mathbf{c}_1\ldots,\mathbf{a}_m,\mathbf{c}_m)\\
    &\leq \mathrm{tr.deg.}_CF + m(|\mathscr{P}|-|\mathscr{P}_0|-1) + m(\dim V+1-\nu),
\end{align*}
 for all $m\in\mathbb{N}$, hence
 \[n \leq |\mathscr{P}|-|\mathscr{P}_0| + \dim V = \dim S+\dim V,\]
 resulting once again in a likely intersection. 

Finally, we consider the case $\nu=\ell_1$.
Consider now the coordinate projection $\mathrm{pr}_{\nu}:\mathbb{A}^n_K\to\mathbb{A}^{\ell_1}_K$ which maps $(\mathbf{a}_i,\mathbf{c}_i)\mapsto\mathbf{c}_i$.
Set $W = \overline{\mathrm{pr}_{\ell_1}(V)}^{\mathrm{Zar}}$. 
As $\Omega_{(\mathscr{P},\mathscr{P}_0),\nu,P_0'}(V)$ is Zariski dense in $V$,  $\mathrm{pr}_{\nu}\left(\Omega_{(\mathscr{P},\mathscr{P}_0),\nu,P_0'}(V)\right)$ is Zariski dense in $W$, which in particular implies that $W$ is defined over $C$. 
By the fibre-dimension theorem and the Zariski density of $\Omega_{(\mathscr{P},\mathscr{P}_0),\nu,P_0'}(V)$ in $V$ we have
 \[\mathrm{tr.deg.}_{F}F(\mathbf{a}_i,\mathbf{c}_i) = \dim V-\dim W,\]
 for all $i\in\mathbb{N}$. 
 Therefore using Ax--Schanuel (Theorem \ref{thm:ax-schanuel}) and (\ref{eq:tauupperbound}) as above gives
 \begin{align*}
    m(n-\ell_1) +1 &\leq \mathrm{tr.deg.}_CF(\boldsymbol{\tau}_1,\ldots,\boldsymbol{\tau}_m,\mathbf{a}_1,\mathbf{c}_1,\ldots,\mathbf{a}_m,\mathbf{c}_m)\\
    &\leq \mathrm{tr.deg.}_CF + m(|\mathscr{P}|-|\mathscr{P}_0|) + m(\dim V-\dim W),
 \end{align*}
for all $m\in\mathbb{N}$. 
Hence
\[n-\ell_1 \leq |\mathscr{P}|-|\mathscr{P}_0| + \dim V-\dim W.\]
If $\dim W=\ell_1$ then we are done, since that implies $n\leq \dim S+\dim V$, giving a likely intersection once more. 

So suppose instead that $\dim W<\ell_1$. 
Consider a point $(\mathbf{a},\mathbf{c})\in\Upsilon_{\Lambda,(\mathscr{P},\mathscr{P}_0),\nu}(V)$. 
Then there is a $\Lambda$-special variety $S$ of special type $(\mathscr{P},\mathscr{P}_0)$ such that $(\mathbf{a},\mathbf{c})\in V\cap S$.
Since $\nu=\ell_1$ then the tuple $\mathbf{c}$ is made up of the values of the coordinates which have fixed value on $S$ and that value is in $C$, and since $S$ is $\Lambda$-special, then $\mathbf{c}$ is a $\Lambda$-special point in $\mathbb{A}^{\ell_1}_K$. 
As no coordinate has a fixed value on $V$ and $V$ is Hodge-generic, then so is $W$, which shows that the weakly special closure of $W$ is $\mathbb{A}^{\ell_1}_K$. 
As $W$ is a proper subvariety of $\mathbb{A}^{\ell_1}_K$, then by modular Mordell--Lang (Theorem \ref{thm:modularmordell-lang}) we get that $W$ cannot have a Zariski dense subset of $\Lambda$-special points, because otherwise $W$ would be $\Lambda$-special, and we would conclude that $V$ is contained in a proper $\Lambda$-special subvariety. 
This shows that there is a proper Zariski closed subset $Z\subsetneq V$ containing $\Upsilon_{\Lambda,(\mathscr{P},\mathscr{P}_0),\nu}(V)$. 
\end{proof}

We are now ready to prove the main theorem.

\begin{proof}[Proof of Theorem \ref{thm:genericmzp}]
    After reindexing the coordinates, we may assume that the first $k$ coordinates have a fixed value on $V$, and no other coordinate is fixed on $V$. 
    Let $F$ be an algebraically closed subfield of $\mathbb{C}$ such that $V$ is defined over $F$ and $\mathrm{tr.deg.}_\mathbb{Q}F$ is finite.
    We proceed by induction on $k$, the case $k=0$ given by Theorem \ref{thm:mainj}.
    
    Suppose now that $k>0$. 
    Say that the values attained by the first $k$ coordinates are $\alpha_1,\ldots,\alpha_k$. 
    Since $V$ is Hodge-generic, then $\alpha_1,\ldots,\alpha_k$ are not singular moduli and they are modularly independent.

    Let $S$ be a special variety with $\dim S < n- \dim V$. 
    Let $(\boldsymbol{\alpha},\boldsymbol{\beta})$ be a generic point of some irreducible component of $V\cap S$ over $F$.
    We remark that while some coordinates may be constant on $S$, these coordinates cannot be among the first $k$. 
    
    If $\alpha_1$ is modularly independent from all the other entries in $(\boldsymbol{\alpha},\boldsymbol{\beta})$, then we consider the projection $\mathrm{pr}:\mathbb{A}_{\mathbb{C}}^n\to\mathbb{A}_{\mathbb{C}}^{n-1}$ which deletes the first coordinate. 
    Clearly $\dim V = \dim\mathrm{pr}(V)$, whereas $\dim \mathrm{pr}(S) = \dim S-1$, so the intersection $\mathrm{pr}(S)\cap\mathrm{pr}(V)$ is still unlikely. 
    We can now apply the induction hypothesis to $\mathrm{pr}(V)$ to conclude that the unlikely locus of $\mathrm{pr}(V)$ is not Zariski dense in $\mathrm{pr}(V)$, and we conclude that there is a proper Zariski closed subset $Z\subsetneq V$ containing all the intersections $V\cap S$, where $\dim S < n-\dim V$ and one of the first $k$ coordinates is modularly independent from all others when restricted to $S$. 

    We now assume that each $\alpha_i$ is modularly dependent with some entry of $\boldsymbol{\beta}$. 
    Consider now the projection $\mathrm{pr}:\mathbb{A}_{\mathbb{C}}^n\to\mathbb{A}_{\mathbb{C}}^{n-k}$ mapping  $(\boldsymbol{\alpha},\boldsymbol{\beta})\to\boldsymbol{\beta}$. 
    We still have $\dim V=\dim\mathrm{pr}(V)$. 
    Set $\Lambda_0=\{\alpha_1,\ldots,\alpha_k\}$ and let $\Lambda\subseteq\mathbb{A}^n_{\mathbb{C}}$ be the structure of finite Hecke-rank determined by $\Lambda_0$. 
    Then $\boldsymbol{\beta}$ is contained in the intersection of $\mathrm{pr}(V)$ with a $\Lambda$-special subvariety of dimension $\leq \dim S-k$. 
    We can now apply Theorem \ref{thm:mainj} to $\mathrm{pr}(V)$ and complete the induction. 
\end{proof}

\begin{remark}
    \label{rem:effective}
As explained in \cite{pila-scanlon}, there are methods to give effective bounds on the degree of the Zariski closure of definable sets in differentially closed fields of characteristic zero, see \cite{binyamini:bezout,freitag-leon:effective,hrushovski-pillay}; in particular we can obtain effective bounds for the degree of the Zariski closure of $\Omega(V)$. 
So, just like in \cite[Corollary 1.1]{pila-scanlon}, the unlikely locus of $V$ is contained in an algebraic set whose degree can be bounded effectively in terms of the degree of $V$ and $n$ (the dimension of the ambient space in which $V$ lies). 
In view of this, there are two natural problems one can consider at this point:
    \begin{enumerate}[(a)]
        \item the Zilber--Pink conjecture is concerned with $\Upsilon(V)$, so we would prefer to have effective bounds on the degree of $\overline{\Upsilon(V)}^{\mathrm{Zar}}$, and
        \item our proof of Theorem \ref{thm:mainj} does not show that $\Omega_\Lambda(V)$ is not Zariski dense in $V$.
    \end{enumerate}

Regarding (a): the main obstacle towards obtaining effective bounds for $\overline{\Upsilon(V)}^{\mathrm{Zar}}$ using the proof we have presented is that this requires an effective form of modular Mordell--Lang.
Even effective Andr\'e--Oort is already a very difficult question, with few cases having yielded a successful answer, some of which we recall now.
The case of plane curves is due to K\"uhne \cite{khune:andre-oort1} (see also \cite{wustholz:andre-oort-pink} for important extra details), and also independently proven by Bilu, Masser and Zannier \cite{bilu-masser-zannier:effectiveao}.
The case of linear subspaces was treated by Bilu and K\"uhne \cite{bilu-kuhne}, and
Binyamini considered \emph{hereditarily degree non-degenerate} hypersurfaces in \cite[Corollary 4]{binyamini:effectiveao}.

Regarding (b): although the proof of Theorem \ref{thm:mainj} does not show $\Omega_\Lambda(V)$ is not Zariski dense, it does show that various subsets of $\Omega_\Lambda(V)$ cannot be Zariski dense in $V$, leaving only the last case, where we looked at the points in $\Upsilon_\Lambda(V)$ instead of $\Omega_\Lambda(V)$ and appealed to modular Mordell--Lang.
To be concrete, suppose $(\mathscr{P},\mathscr{P}_0)$ is such that $|\mathscr{P}|-|\mathscr{P}_0| < n - \dim V$. 
The proof of Theorem \ref{thm:mainj} shows that if $\Omega_{(\mathscr{P},\mathscr{P}_0),\nu,P_0'}(V)$ were Zariski dense in $V$, then the following conditions are necessary:
\begin{enumerate}[(i)]
    \item $\nu>0$, 
    \item $\ell_1 = \nu$, where $\ell_1 = |P_0| - |P_0'|$, and 
    \item $\dim W < \ell_1$,  where $W$ is the Zariski closure of the projection of $V$ to $\mathbb{A}^{\ell_1}_K$. 
\end{enumerate}
In particular, if we only consider the unlikely intersections between $V$ and strongly special varieties (as it is done in \cite[Corollary 1.1]{pila-scanlon}), then we do not need to consider a set $\Omega_{(\mathscr{P},\mathscr{P}_0),\nu,P_0'}(V)$ satisfying the above conditions (because $\ell_1$ equals 0 in this case), and so this part of the unlikely locus is contained in a proper algebraic subset of $V$ whose degree is effectively bounded. 
\end{remark}

\bibliographystyle{alphaurl}
\bibliography{references}{}

\begin{thebibliography}{MPT19}

\bibitem[AD24]{aslanyan-daw}
Vahagn Aslanyan and Christopher Daw.
\newblock A note on unlikely intersections in {S}himura varieties.
\newblock {\em J. Number Theory}, 260:212--222, 2024.
\newblock \href {https://doi.org/10.1016/j.jnt.2024.02.015} {\path{doi:10.1016/j.jnt.2024.02.015}}.

\bibitem[AK22]{aslanyan-kirby}
Vahagn Aslanyan and Jonathan Kirby.
\newblock Blurrings of the {$j$}-function.
\newblock {\em Q. J. Math.}, 73(2):461--475, 2022.
\newblock \href {https://doi.org/10.1093/qmath/haab037} {\path{doi:10.1093/qmath/haab037}}.

\bibitem[Asl21a]{aslanyan:strongminj}
Vahagn Aslanyan.
\newblock Ax-{S}chanuel and strong minimality for the {$j$}-function.
\newblock {\em Ann. Pure Appl. Logic}, 172(1):Paper No. 102871, 24, 2021.
\newblock \href {https://doi.org/10.1016/j.apal.2020.102871} {\path{doi:10.1016/j.apal.2020.102871}}.

\bibitem[Asl21b]{aslanyan:atypical}
Vahagn Aslanyan.
\newblock Some remarks on atypical intersections.
\newblock {\em Proc. Amer. Math. Soc.}, 149(11):4649--4660, 2021.
\newblock \href {https://doi.org/10.1090/proc/15611} {\path{doi:10.1090/proc/15611}}.

\bibitem[Asl22]{aslanyan:adequate}
Vahagn Aslanyan.
\newblock Adequate predimension inequalities in differential fields.
\newblock {\em Ann. Pure Appl. Logic}, 173(1):Paper No. 103030, 42, 2022.
\newblock \href {https://doi.org/10.1016/j.apal.2021.103030} {\path{doi:10.1016/j.apal.2021.103030}}.

\bibitem[BD25]{barroero-dill:distinguishedcategories}
Fabrizio Barroero and Gabriel~A. Dill.
\newblock Distinguished categories and the {Z}ilber-{P}ink conjecture.
\newblock {\em Amer. J. Math.}, 147(3):715--778, 2025.
\newblock \href {https://doi.org/10.1353/ajm.2025.a961345} {\path{doi:10.1353/ajm.2025.a961345}}.

\bibitem[BD26]{barroero-dill:heckeorbits}
Fabrizio Barroero and Gabriel~A. Dill.
\newblock Hecke orbits and the {M}ordell-{L}ang conjecture in distinguished categories.
\newblock {\em J. Number Theory}, 287:154--190, 2026.
\newblock \href {https://doi.org/10.1016/j.jnt.2026.01.004} {\path{doi:10.1016/j.jnt.2026.01.004}}.

\bibitem[Bin17]{binyamini:bezout}
Gal Binyamini.
\newblock Bezout-type theorems for differential fields.
\newblock {\em Compos. Math.}, 153(4):867--888, 2017.
\newblock \href {https://doi.org/10.1112/S0010437X17007035} {\path{doi:10.1112/S0010437X17007035}}.

\bibitem[Bin20]{binyamini:effectiveao}
Gal Binyamini.
\newblock Some effective estimates for {A}ndr\'e-{O}ort in {$Y(1)^n$}.
\newblock {\em J. Reine Angew. Math.}, 767:17--35, 2020.
\newblock With an appendix by Emmanuel Kowalski.
\newblock \href {https://doi.org/10.1515/crelle-2019-0028} {\path{doi:10.1515/crelle-2019-0028}}.

\bibitem[BK20]{bilu-kuhne}
Yuri Bilu and Lars K\"uhne.
\newblock Linear equations in singular moduli.
\newblock {\em Int. Math. Res. Not. IMRN}, (21):7617--7643, 2020.
\newblock \href {https://doi.org/10.1093/imrn/rny216} {\path{doi:10.1093/imrn/rny216}}.

\bibitem[BMZ13]{bilu-masser-zannier:effectiveao}
Yuri Bilu, David Masser, and Umberto Zannier.
\newblock An effective ``theorem of {A}ndr\'e'' for {$CM$}-points on a plane curve.
\newblock {\em Math. Proc. Cambridge Philos. Soc.}, 154(1):145--152, 2013.
\newblock \href {https://doi.org/10.1017/S0305004112000461} {\path{doi:10.1017/S0305004112000461}}.

\bibitem[Daw25]{daw:unlikelyintersectionsshimuravarieties}
Christopher Daw.
\newblock Unlikely intersections in {S}himura varieties and beyond: a survey, 2025.
\newblock \href {https://arxiv.org/abs/2506.02900} {\path{arXiv:2506.02900}}.

\bibitem[DO25]{daw-orr:zpproductmodularcurves}
Christopher Daw and Martin Orr.
\newblock Zilber--{P}ink in a product of modular curves assuming multiplicative degeneration.
\newblock {\em Duke Math. J.}, 174(13):--, 2025.
\newblock \href {https://doi.org/10.1215/00127094-2025-0011} {\path{doi:10.1215/00127094-2025-0011}}.

\bibitem[DR18]{daw-ren}
Christopher Daw and Jinbo Ren.
\newblock Applications of the hyperbolic {A}x-{S}chanuel conjecture.
\newblock {\em Compos. Math.}, 154(9):1843--1888, 2018.
\newblock \href {https://doi.org/10.1112/s0010437x1800725x} {\path{doi:10.1112/s0010437x1800725x}}.

\bibitem[FLS16]{freitag-leon:effective}
James Freitag and Omar Le\'on~S\'anchez.
\newblock Effective uniform bounding in partial differential fields.
\newblock {\em Adv. Math.}, 288:308--336, 2016.
\newblock \href {https://doi.org/10.1016/j.aim.2015.10.013} {\path{doi:10.1016/j.aim.2015.10.013}}.

\bibitem[FS18]{freitag-scanlon:strongminj}
James Freitag and Thomas Scanlon.
\newblock Strong minimality and the {$j$}-function.
\newblock {\em J. Eur. Math. Soc. (JEMS)}, 20(1):119--136, 2018.
\newblock \href {https://doi.org/10.4171/JEMS/761} {\path{doi:10.4171/JEMS/761}}.

\bibitem[HP00]{hrushovski-pillay}
Ehud Hrushovski and Anand Pillay.
\newblock Effective bounds for the number of transcendental points on subvarieties of semi-abelian varieties.
\newblock {\em Amer. J. Math.}, 122(3):439--450, 2000.
\newblock \href {https://doi.org/10.1353/ajm.2000.0020} {\path{doi:10.1353/ajm.2000.0020}}.

\bibitem[HP12]{habegger-pila}
Philipp Habegger and Jonathan Pila.
\newblock Some unlikely intersections beyond {A}ndr\'{e}-{O}ort.
\newblock {\em Compos. Math.}, 148(1):1--27, 2012.
\newblock \href {https://doi.org/10.1112/S0010437X11005604} {\path{doi:10.1112/S0010437X11005604}}.

\bibitem[HP16]{habegger-pila:o-min}
Philipp Habegger and Jonathan Pila.
\newblock O-minimality and certain atypical intersections.
\newblock {\em Ann. Sci. \'{E}c. Norm. Sup\'{e}r. (4)}, 49(4):813--858, 2016.
\newblock \href {https://doi.org/10.24033/asens.2296} {\path{doi:10.24033/asens.2296}}.

\bibitem[Kli23]{klingler:hodgetheoryicm}
Bruno Klingler.
\newblock Hodge theory, between algebraicity and transcendence.
\newblock In {\em I{CM}---{I}nternational {C}ongress of {M}athematicians. {V}ol. 3. {S}ections 1--4}, pages 2250--2284. EMS Press, Berlin, 2023.
\newblock \href {https://doi.org/10.4171/icm2022/158} {\path{doi:10.4171/icm2022/158}}.

\bibitem[KT25]{klingler-tayou:zilberpink}
Bruno Klingler and Salim Tayou.
\newblock The {Z}ilber-{P}ink conjecture for varieties not defined over {$\overline{\mathbb {Q}}$}, 2025.
\newblock \href {https://arxiv.org/abs/2504.00865} {\path{arXiv:2504.00865}}.

\bibitem[K{\"u}h12]{khune:andre-oort1}
Lars K{\"u}hne.
\newblock An effective result of {A}ndr\'e-{O}ort type.
\newblock {\em Ann. of Math. (2)}, 176(1):651--671, 2012.
\newblock \href {https://doi.org/10.4007/annals.2012.176.1.13} {\path{doi:10.4007/annals.2012.176.1.13}}.

\bibitem[Lan87]{lang:elliptic}
Serge Lang.
\newblock {\em Elliptic functions}, volume 112 of {\em Graduate Texts in Mathematics}.
\newblock Springer-Verlag, New York, second edition, 1987.
\newblock With an appendix by J. Tate.
\newblock \href {https://doi.org/10.1007/978-1-4612-4752-4} {\path{doi:10.1007/978-1-4612-4752-4}}.

\bibitem[Mah69]{mahler}
Kurt Mahler.
\newblock On algebraic differential equations satisfied by automorphic functions.
\newblock {\em J. Austral. Math. Soc.}, 10:445--450, 1969.
\newblock \href {https://doi.org/10.1017/S1446788700007709} {\path{doi:10.1017/S1446788700007709}}.

\bibitem[Mar06]{marker:modeltheorydcf}
David Marker.
\newblock {\em Model Theory of Differential Fields}, volume~5 of {\em Lecture Notes in Logic}, pages 38--113.
\newblock Association for Symbolic Logic, La Jolla, CA; A K Peters, Ltd., Wellesley, MA, second edition, 2006.
\newblock \href {https://doi.org/10.1017/9781316716991.003} {\path{doi:10.1017/9781316716991.003}}.

\bibitem[Mas03]{masser}
David Masser.
\newblock Heights, transcendence, and linear independence on commutative group varieties.
\newblock In {\em Diophantine approximation ({C}etraro, 2000)}, volume 1819 of {\em Lecture Notes in Math.}, pages 1--51. Springer, Berlin, 2003.
\newblock \href {https://doi.org/10.1007/3-540-44979-5_1} {\path{doi:10.1007/3-540-44979-5_1}}.

\bibitem[MPT19]{mpt}
Ngaiming Mok, Jonathan Pila, and Jacob Tsimerman.
\newblock Ax-{S}chanuel for {S}himura varieties.
\newblock {\em Ann. of Math. (2)}, 189(3):945--978, 2019.
\newblock \href {https://doi.org/10.4007/annals.2019.189.3.7} {\path{doi:10.4007/annals.2019.189.3.7}}.

\bibitem[Pap24]{papas:zpnmultiplicativedeg}
Georgios Papas.
\newblock {Z}ilber-{P}ink in ${Y}(1)^n$: Beyond multiplicative degeneration, 2024.
\newblock To appear in Israel J. Math.
\newblock \href {https://arxiv.org/abs/2402.09487} {\path{arXiv:2402.09487}}.

\bibitem[Pil11]{pila:andre-oort}
Jonathan Pila.
\newblock O-minimality and the {A}ndr\'{e}-{O}ort conjecture for {$\mathbb{C}^n$}.
\newblock {\em Ann. of Math. (2)}, 173(3):1779--1840, 2011.
\newblock \href {https://doi.org/10.4007/annals.2011.173.3.11} {\path{doi:10.4007/annals.2011.173.3.11}}.

\bibitem[Pil14]{pila:specialpts}
Jonathan Pila.
\newblock Special point problems with elliptic modular surfaces.
\newblock {\em Mathematika}, 60(1):1--31, 2014.
\newblock \href {https://doi.org/10.1112/S0025579313000168} {\path{doi:10.1112/S0025579313000168}}.

\bibitem[Pil17]{pila:fermat}
Jonathan Pila.
\newblock On a modular {F}ermat equation.
\newblock {\em Comment. Math. Helv.}, 92(1):85--103, 2017.
\newblock \href {https://doi.org/10.4171/CMH/407} {\path{doi:10.4171/CMH/407}}.

\bibitem[Pil22]{pila_2022}
Jonathan Pila.
\newblock {\em Point-counting and the {Z}ilber-{P}ink conjecture}, volume 228 of {\em Cambridge Tracts in Mathematics}.
\newblock Cambridge University Press, Cambridge, 2022.
\newblock \href {https://doi.org/10.1017/9781009170314} {\path{doi:10.1017/9781009170314}}.

\bibitem[PS21]{pila-scanlon}
Jonathan Pila and Thomas Scanlon.
\newblock Effective transcendental {Z}ilber-{P}ink for variations of {H}odge structures, 2021.
\newblock \href {https://arxiv.org/abs/2105.05845} {\path{arXiv:2105.05845}}.

\bibitem[PT16]{pila-tsimerman:ax-schanuel}
Jonathan Pila and Jacob Tsimerman.
\newblock Ax-{S}chanuel for the {$j$}-function.
\newblock {\em Duke Math. J.}, 165(13):2587--2605, 2016.
\newblock \href {https://doi.org/10.1215/00127094-3620005} {\path{doi:10.1215/00127094-3620005}}.

\bibitem[RY24]{richard-yafaev:heightfunctions}
Rodolphe Richard and Andrei Yafaev.
\newblock Height functions on {H}ecke orbits and the generalised {A}ndr\'e-{P}ink-{Z}annier conjecture.
\newblock {\em Compos. Math.}, 160(11):2531--2584, 2024.
\newblock \href {https://doi.org/10.1112/S0010437X2400722X} {\path{doi:10.1112/S0010437X2400722X}}.

\bibitem[RY25]{richard-yafaev:generalizedAPZ}
Rodolphe Richard and Andrei Yafaev.
\newblock Generalised {A}ndr\'e-{P}ink-{Z}annier conjecture for {S}himura varieties of {A}belian type.
\newblock {\em Publ. Math. Inst. Hautes \'Etudes Sci.}, 141:249--331, 2025.
\newblock \href {https://doi.org/10.1007/s10240-025-00154-4} {\path{doi:10.1007/s10240-025-00154-4}}.

\bibitem[TZ12]{tent-ziegler}
Katrin Tent and Martin Ziegler.
\newblock {\em A course in model theory}, volume~40 of {\em Lecture Notes in Logic}.
\newblock Association for Symbolic Logic, La Jolla, CA; Cambridge University Press, Cambridge, 2012.
\newblock \href {https://doi.org/10.1017/CBO9781139015417} {\path{doi:10.1017/CBO9781139015417}}.

\bibitem[W{\"u}s14]{wustholz:andre-oort-pink}
Gisbert W{\"u}stholz.
\newblock A note on the conjectures of {A}ndr\'e-{O}ort and {P}ink with an appendix by {L}ars {K}\"uhne.
\newblock {\em Bull. Inst. Math. Acad. Sin. (N.S.)}, 9(4):735--779, 2014.

\bibitem[Zag08]{zagier:elliptic}
Don Zagier.
\newblock Elliptic modular forms and their applications.
\newblock In {\em The 1-2-3 of modular forms}, Universitext, pages 1--103. Springer, Berlin, 2008.
\newblock \href {https://doi.org/10.1007/978-3-540-74119-0_1} {\path{doi:10.1007/978-3-540-74119-0_1}}.

\end{thebibliography}

\end{document}